\documentclass{amsart}
\usepackage{amssymb}
\usepackage{amsthm}
\usepackage{epsfig,color}
\usepackage{enumerate}

\newcommand{\mc}[1]{\mathcal{#1}}
\newcommand{\twid}[1]{\tilde{#1}}

\def\co{\colon\thinspace}
\def\R{\mathbb{R}}

\def\N{\mathbb{N}}

\newcommand{\extend}[1]{\check{{#1}}}
\newcommand{\straightedge}[1]{{#1}_{\mc{T}}}

\newcommand{\skel}{\mathrm{Skel}}
\newcommand{\skelmap}{\straightedge{\ddot{{\phi}}}}
\newcommand{\skelfill}{\ensuremath{\skel(\straightedge{\extend{\phi}})}}
\newcommand{\Lk}{\mathrm{Lk}}
\def\Hbound {\partial_{\mc{H}}X}
\newcommand{\stab}{\text{Stab}}

\def\logConst {\lambda(K+1)}
\def\Const {(3 + 40\delta + 2 \log_2(\logConst))}

\newtheorem{theorem}{Theorem}[section]
\newtheorem{lemma}[theorem]{Lemma}
\newtheorem{corollary}[theorem]{Corollary}
\newtheorem{proposition}[theorem]{Proposition}

\newtheorem{conjecture}[theorem]{Conjecture}
\newtheorem{observation}[theorem]{Observation}
\newtheorem{case}{Case}
\newtheorem{subcase}{Case}[case]

\newtheorem{claim}{Claim}[theorem]
\newtheorem{subclaim}{Subclaim}[claim]

\theoremstyle{definition}
\newtheorem{remark}[theorem]{Remark}
\newtheorem{definition}[theorem]{Definition}
\newtheorem{example}[theorem]{Example}

\begin{document}

\title{Residual finiteness, QCERF, and fillings of hyperbolic groups}

\author{Ian Agol}
\address{
   University of California, Berkeley,
   970 Evans Hall \#3840,
   Berkeley, CA 94720-3840}
\email{ianagol@math.berkeley.edu}

\author{Daniel Groves}
\address{Department of Mathematics, Statistics, and Computer Science,
University of Illinois at Chicago,
322 Science and Engineering Offices (M/C 249),
851 S. Morgan St.,
Chicago, IL 60607-7045}
\email{groves@math.uic.edu}

\author{Jason Fox Manning}
\address{Department of Mathematics, SUNY at Buffalo,
Buffalo, NY 14260-2900}
\email{j399m@buffalo.edu}

\thanks{The first author was partially supported by NSF grant DMS-0504975, and the
Guggenheim foundation.  The second 
author was supported by NSF grant DMS 0504251.}

\begin{abstract}
We prove that if every hyperbolic group is residually finite, then
every quasi-convex subgroup of every hyperbolic group is
separable. The main tool is relatively hyperbolic Dehn filling.
\end{abstract} 

\maketitle

\tableofcontents

A group $G$ is \emph{residually finite} (or \emph{RF}) if
for every $g\in G\smallsetminus\{1\}$, there is some finite group $F$
and an epimorphism $\phi\co G\to F$ so that $\phi(g)\neq 1$.
In more sophisticated language $G$ is RF if and only if
the trivial subgroup is closed in the profinite topology on $G$.

If $H<G$, then \emph{$H$ is separable} if for every $g\in
G\smallsetminus H$, there is some finite group $F$ and an epimorphism
$\phi\co G\to F$ so that $\phi(g)\notin \phi(H)$.  
Equivalently, the subgroup $H$
is separable in $G$ if it is closed in the profinite topology on
$G$.  

If every finitely generated subgroup of $G$ is separable,
$G$ is said to be \emph{LERF}, or \emph{subgroup separable}.  
If $G$ is hyperbolic, and every quasi-convex subgroup of
$G$ is separable, we say that $G$ is \emph{QCERF}.

In this paper, we show that if every hyperbolic group is RF, then
every hyperbolic group is QCERF.
\begin{theorem}\label{t:main}
If all hyperbolic groups are residually finite, then every
quasi-convex subgroup of a hyperbolic group is separable.
\end{theorem}

\begin{remark}
Theorem \ref{t:main} states that the existence of a non-residually
finite hyperbolic group is equivalent to the existence of a
non-separable quasi-convex subgroup of some hyperbolic group.  This
equivalence was guessed by Dani Wise in \cite{wise:polygons}.
Wise (\emph{op. cit.}) and Minasyan 
\cite{minasyan:subsetgferf} noticed independently that
an argument of Long and Niblo
\cite{longniblo:subgroupsep} can be used to show that residual
finiteness for all hyperbolic groups implies separability of all
\emph{almost malnormal} quasi-convex subgroups.

In the other direction,  Kapovich and Wise show in \cite{kapwise} that
if every hyperbolic group has a finite index subgroup, then every
hyperbolic group is residually finite.  Together with our result, this
gives the statement:  If every hyperbolic group has a finite index
subgroup, then every hyperbolic group is QCERF.
\end{remark}

To prove Theorem \ref{t:main}, for a hyperbolic group $G$ with 
quasi-convex subgroup $H<G$ and $g\in G-H$ an element to separate, 
we would like to find a hyperbolic quotient $\varphi: G\to K$, such
that $\varphi(H)< K$ is finite, and $\varphi(g)\notin \varphi(H)$. 
Then since $K$ is assumed to be hyperbolic and therefore residually
finite, we may separate $g$ from $H$. One natural way to attempt
to find such a quotient $\varphi$ would be to ``kill" a large finite-index
normal subgroup $H'\lhd H$, and hope that the quotient of $G$ is still hyperbolic
and that $H$ projects to $H/H'$. 
This actually works if $H$ is malnormal in $G$. 
The difficulty with this procedure if $H$ is not malnormal is that one must make sure that for
any $k\in G-H$, if $U=H\cap H^k\neq \emptyset$, then 
$H' \cap U = H' \cap U^{k^{-1}} \subset H$, otherwise killing
$H'$ would force a larger subset of $H$ to be killed. Thus, we
need to take into account intersections between $H$ and its conjugates,
which motivates considering the following definitions.

Let $H^g=gHg^{-1}$. 
The following was defined in \cite{gmrs}.

\begin{definition}
Let $H$ be a subgroup of a group $G$. The elements 
$\{g_i \mid 1 \leq i \leq n\}$ of $G$ are said to be \emph{essentially distinct} if 
$g_i H  \neq g_j H $ for $i \neq j$. Conjugates $\{H^{g_i}\mid 1\leq
i\leq n\}$ of  
$H$ by essentially distinct elements are called \emph{essentially distinct conjugates}. 
\end{definition}
It should be remarked that essentially distinct conjugates may coincide
if $H$ is not equal to its own normalizer.

\begin{definition}\label{d:height} The {\em height} of an infinite subgroup $H < G$ is $n$ if there 
exists a collection of $n$ essentially distinct conjugates of $H$ such that the intersection 
of all the elements of the collection is infinite and $n$ is maximal possible. The height of a finite subgroup is $0$. 
\end{definition}
For example, an infinite, malnormal subgroup has height $1$, whereas
an infinite normal subgroup has height equal to its index.
The most relevant result about height for our purposes is the
following theorem of Gitik, Mitra, Rips, and Sageev.
\begin{theorem}\label{t:finiteheight} \cite{gmrs}
A quasi-convex subgroup of a hyperbolic group has finite height.
\end{theorem}

  The proof of Theorem \ref{t:main} will be by induction on
height, using the following theorem, which is the main technical
result of this paper.

\begin{theorem} \label{t:technical}
Let $G$ be a torsion-free residually finite hyperbolic group, let
$H$ be a quasi-convex subgroup of $G$ of height $k$, 
and let $g \in G \smallsetminus H$.  
There is a quotient $\eta : G \to \bar{G}$ so that;
\begin{enumerate}
\item $\bar{G}$ is hyperbolic;
\item $\eta(H)$ is quasi-convex in $\bar{G}$;
\item $\eta(g) \not\in \eta(H)$; and
\item the height of $\eta(H)$ in $\bar{G}$ is at most $k-1$.
\end{enumerate}
\end{theorem}

\begin{proof}[Proof of Theorem \ref{t:main} from Theorem \ref{t:technical}]
Let $G$ be a hyperbolic group and $H$ a quasi-convex subgroup of
$G$ of height $k$.  We prove that $H$ is separable by induction on height.

The base case is when $H$ has height zero, which means $H$ is
finite.  Since $G$ is residually finite it is straightforward to separate
any $g \in G \smallsetminus H$ from the finite set $H$.

Assume that $k \ge 1$. 
We claim that it suffices to prove $H$ is separable in the special
case that $G$ is torsion-free.  Indeed, let $G_0 
\le G$ be a torsion-free subgroup of finite-index.  Such a $G_0$ 
exists because $G$ is residually-finite and $G$ has only finitely many
conjugacy classes of torsion elements (see, e.g. \cite{brady:finite}).  
Further, let $H_0 = G_0 \cap
H$.  An elementary argument shows that the height of $H_0$ is at
most $k$.  Equally, if $H_0$ is separable in $G_0$ then $H$ is
separable in $G$.  To see this, note that since $G_0$ is of finite-index
in $G$, the profinite topology on $G_0$ coincides with the subspace topology
induced by the profinite topology on $G$.  Thus, if $H_0$ is closed in the profinite topology
on $G_0$ then it is closed in the profinite topology on $G$.  The subgroup $H$
is a finite union of cosets of $H_0$, and is therefore closed in the profinite topology on $G$.

We have now reduced to the case that $G$ is torsion-free.  Let
$g \in G \smallsetminus H$.  By Theorem
\ref{t:technical} there is a hyperbolic quotient $\bar{G}$ of $G$ which
separates $g$ from $H$, and the image of $H$ in $\bar{G}$ is 
quasi-convex and has height at most $k-1$.  Theorem \ref{t:main} follows
by induction.
\end{proof}

The remainder of the paper is devoted to the proof of Theorem
\ref{t:technical}.

\textbf{Acknowledgments:} The first author thanks Kevin Whyte for a
useful conversation.  We would also like to thank Igor Belegradek for
corrections to an earlier version of the conclusion, and
Eduardo Martinez-Pedroza, who pointed out a serious error in an earlier version of
the proof of Lemma \ref{l:technical2}.

\section{The cusped space of a relatively hyperbolic group}
In this section we briefly recall the main constructions of
\cite{rhds}.  Briefly, given a finitely generated group $G=\langle
S\rangle$ and a finite collection of finitely generated subgroups 
$\mc{P}$, we 
build a ``cusped space'' $X(G,\mc{P},S)$ by first forming the Cayley graph
of $G$ and then gluing a ``horoball'' onto each translate of an
element of $\mc{P}$.

\begin{definition}\label{d:combhoro}
Let $\Gamma$ be any $1$-complex.
The \emph{combinatorial horoball based on $\Gamma$}, denoted
$\mc{H}(\Gamma)$, is the $2$-complex formed as follows:
\begin{itemize}
\item $\mc{H}^{(0)}= \Gamma^{(0)}\times \left( \{0\}\cup \N \right)$
\item $\mc{H}^{(1)}$ contains the following three types of edges.  The
  first two types are called \emph{horizontal}, and the last type is
  called \emph{vertical}.
\begin{enumerate}
\item If $e$ is an edge of $\Gamma$ joining $v$ to $w$ then there is a
  corresponding edge $\bar{e}$ connecting $(v,0)$ to $(w,0)$.
\item If $k>0$ and $0<d_{\Gamma}(v,w)\leq 2^k$, then there is a single edge
  connecting $(v,k)$ to $(w,k)$.
\item If $k\geq 0$ and $v\in \Gamma^{(0)}$, there is an edge  joining
  $(v,k)$ to $(v,k+1)$. 
\end{enumerate}
\item $\mc{H}^{(2)}$ contains $2$-cells (described explicitly in
  \cite[Definition 3.1]{rhds}) which ensure that $\mc{H}$ satisfies a linear
  isoperimetric inequality, with constant independent of $\Gamma$.
\end{itemize}
\end{definition}

\begin{remark}\label{r:subset}
As the full subgraph of $\mc{H}(\Gamma)$ containing the vertices
$\Gamma^{(0)}\times\{0\}$ is isomorphic to $\Gamma$, we may think of
$\Gamma$ as a subset of $\mc{H}(\Gamma)$.
\end{remark} 

\begin{definition} \label{d:Depth1}
Let $\Gamma$ be a graph and $\mc{H}(\Gamma)$ the associated
combinatorial horoball.  Define a \emph{depth} function
\[D: \mc{H}(\Gamma) \to [0,\infty)\]
which satisfies:
\begin{enumerate}
\item $D(x)=0$ if $x\in \Gamma$,
\item $D(x)=k$ if $x$ is a vertex $(v,k)$, and
\item $D$ restricts to an affine function on each $1$-cell and on each
  $2$-cell.
\end{enumerate}
\end{definition}

\begin{definition} \label{d:relpres}
[Osin]
Suppose that $G$ is generated by $S$ with respect to
$\{ H_\lambda \}_{\lambda \in \Lambda}$.  This means $G$ is a quotient
of 
\[ F = 	F(S) \ast \left( \ast_{\lambda \in \Lambda} H_\lambda \right),	\]
where $F(S)$ is the free group on the alphabet $S$.
Suppose that $N$ is the kernel of the canonical quotient map
from $F$ to $G$.  If $N$ is the normal closure of the set $\mc{R}$ then
we say that
\[	\langle S, \{ H_\lambda \}_{\lambda \in \Lambda} \mid
\mc{R} \rangle	,	\]
is a {\em relative presentation} for $G$ with respect to 
$\{ H_\lambda \}_{\lambda \in \Lambda}$.

We say that $G$ is {\em finitely presented relative to
$\{ H_\lambda \}_{\lambda \in \Lambda}$} if we can choose
$\mc{R}$ and $S$ to be finite.
\end{definition}

\begin{definition}\label{d:X1} \emph{The cusped space $ X(G,\mc{P},S)$.}
Let $G$ be a finitely generated group which is finitely presented
relative to
$\mc{P}=\{P_1,\ldots,P_m \}$, a 
family of finitely generated subgroups of $G$. Let $S$ be a
generating set for $G$ so that $P_i\cap S$ generates $P_i$ for each
$i\in \{1,\ldots,m\}$.  
For each $i\in \{1,\ldots,m\}$, let $T_i$ be a left transversal for
$P_i$ (i.e. a collection of representatives for left cosets of $P_i$ in $G$
which contains exactly one element of each left coset). Let
$\Gamma = \Gamma(G,S)$ be the Cayley graph of $G$.  
To $\Gamma$, equivariantly attach $2$-cells coming
from the finite relative presentation to obtain a $2$-complex
$\overline {\Gamma}$.

For each $i$, and each $t\in T_i$, let $\Gamma_{i,t}$ be the full
subgraph of the Cayley graph $\Gamma(G,S)$ which contains $tP_i$.
Each $\Gamma_{i,t}$ is isomorphic to the Cayley graph of $P_i$ with
respect to the generators $P_i\cap S$.
Then we define
\[ X(G,\mc{P},S) = \overline{\Gamma} \cup (\cup \{\mc{H}(\Gamma_{i,t})\mid 1\leq i\leq m,
t\in T_i\}),\]
where the graphs $\Gamma_{i,t}\subset \Gamma(G,S)$ and
$\Gamma_{i,t}\subset \mc{H}(\Gamma_{i,t})$ are identified as suggested
in Remark \ref{r:subset}.  
\end{definition}

\begin{definition}
A \emph{horoball} of $X(G,\mc{P},S)$ is the subgraph
$\mc{H}(\Gamma_{i,t})$ for some $i$ and $t$.  For $l\in \N$, an
$l$-horoball is the subgraph of $\mc{H}(\Gamma_{i,t})$ whose vertices
are all of distance at least $l$ from the Cayley graph $\Gamma$.
\end{definition}

\begin{remark}\label{r:orderedpair}
Once a horoball is specified, the vertex of $X(G,\mc{P},S)$ 
connected by a vertical geodesic of length $n$ to the group element
$g$ can be conveniently referred to by the ordered pair $(g,n)$, and
we will often do so.
\end{remark}

\begin{remark}
Whenever $X(G,\mc{P},S)$ is to be thought of as a metric space, we
will always implicitly ignore the $2$-cells, and regard
$\mc{H}(\Gamma)^{(1)}$ as a metric graph with all edges of length
one.
\end{remark}

Relative hyperbolicity was
first defined by Gromov in \cite{gromov:wordhyperbolic}.  We use the
following characterization (See \cite[Section 3]{rhds} for this
characterization and others):
\begin{proposition}
Let $G$ be a finitely generated group, and let $\mc{P}$ be a finite
collection of finitely generated subgroups.  The following are
equivalent:
\begin{enumerate}
\item $G$ is hyperbolic relative to $\mc{P}$ in the sense of Gromov.
\item The space $X(G,\mc{P},S)$ is Gromov hyperbolic for some finite
  generating set $S$. 
\item The space $X(G,\mc{P},S)$ satisfies a linear isoperimetric inequality.
\end{enumerate}
\end{proposition}

Most of our geometric arguments therefore take place in some cusped
space $X=X(G,\mc{P},S)$.  For most of the paper, we will work either
with arbitrary geodesics in this space, or with \emph{regular}
geodesics, i.e. geodesics whose intersection with any horoball is
vertical except possibly for a single horizontal sub-segment.  In
Subsection \ref{ss:heightdecrease}, we will need to use paths between
points in $X$ (and sometimes $\partial X$) whose behavior is even more
controlled.  These are the \emph{preferred paths} of \cite{rhds}, and
we refer to that paper for a detailed discussion.

\section{Filling hyperbolic and relatively hyperbolic groups}

Let $G$ be hyperbolic relative to a finite collection $\mc{P} = \{ P_1, \ldots , P_m \}$,
as in the previous section.
A \emph{filling} of $G$ is determined by a choice of subgroups
$N_j\lhd P_j$, called \emph{filling kernels}; we write the quotient
after filling as $G(N_1,\ldots,N_m)$.  If $S$ is a generating set for
$G$ which contains generating sets for each $P_i$, then for each $i$ we define the
\emph{algebraic slope length}, denoted $|N_i|_{P_i}$, to be the length of the
shortest nontrivial element of $N_i$, measured in the generators
$S\cap P_i$.

We collect here some results about filling from \cite{rhds} (see
also \cite{osin:peripheral}):
\begin{theorem}\label{t:rhds}
Let $G$ be a torsion-free group, which is hyperbolic
relative to a collection $\mc{P} = \{ P_1, \ldots , P_m \}$
of finitely generated subgroups.  Suppose that $S$
is a generating set for $G$ so that for each $1 \le i \le m$ 
we have $P_i = \langle P_i \cap S \rangle$.
Let $F\subseteq G$ be a finite set.

There exists a constant $B$ depending only on 
$G$, $\mc{P}$, $S$, and $F$ so that for any collection $\{ N_i \}_{i=1}^m$ 
of subgroups satisfying
\begin{itemize}
\item $N_i \unlhd P_i$ and
\item $| N_i |_{P_i} \ge B$,
\end{itemize}
then the following hold, where $K$ is the normal closure in $G$ of
$N_1\cup \cdots \cup N_m$, and $G(N_1,\ldots,N_m) = G/K$:
\begin{enumerate}
\item\label{b:injective} \cite[Theorem 9.1]{rhds} The map $P_i/N_i\xrightarrow{\iota_i}G(N_1,\ldots,N_m)$ given by $pN_i\mapsto pK$
  is injective for each $i$.
\item\label{b:relhyp} \cite[Theorem 7.2]{rhds} $G(N_1,\ldots,N_m)$ is hyperbolic relative to the collection
  $\mc{Q}=\{\iota_i(P_i/N_i)\mid 1\leq i\leq m\}$.
\item\label{b:inject} \cite[Corollary 9.7]{rhds} The projection from
  $G$ to $G(N_1,\ldots,N_m)$ is injective on $F$.
\end{enumerate}
\end{theorem}

The following lemma is needed in the proof of Proposition
\ref{p:uniformhyp}.  Its statement involves both `relative Dehn
functions' and the `coned-off Cayley complex' of a relatively hyperbolic
group.  We refer the reader to \cite[Section 2.3]{osin:relhypbook}
and \cite[Definition 2.47]{rhds} for the definitions.

\begin{lemma} \label{l:regulator}
Suppose that $G$ is hyperbolic relative to $\mc{P}$ and 
that 
\[	\langle X, \{ P_\lambda \}_{\lambda \in \Lambda} \mid
\mc{R} \rangle, 	\]
is a finite relative presentation for $G$ with respect to $\mc{P}$.
Let $M = \max_{r \in \mc{R}} |r|$.

Suppose further that
that $G$ has a linear relative Dehn function with constant $K$.
Then the coned-off Cayley complex of $G$ with respect to $\mc{P}$
has a linear isoperimetric function with constant at most $(M+1)K +1$.
\end{lemma}
\begin{proof}
Let $\hat{C}$ be the coned-off Cayley complex.
Start with a loop $c$ in the $1$-skeleton of $\hat{C}$.  We may clearly
assume that $c$ is embedded.  

Let $\Gamma$ be the Cayley graph of $G$ with respect to the generating set $X \cup (\cup_{\lambda} P_\lambda \smallsetminus \{ 1 \})$.  Any loop in $\Gamma$ can be filled with a disk whose $2$-cells
have boundary labelled either by elements of $\mc{R}$ (`$\mc{R}$-cells') or by a relation in one of the $P_\lambda$ (`$\mc{P}$-cells').

We replace the loop $c$ by a loop $c'$ in $\Gamma$ by taking each
sub-segment of $c$ of length $2$ which has a cone point as its
midpoint and replacing it with the corresponding edge of $\Gamma$.

Clearly $|c'| \le |c|$.  There is therefore some filling of $c'$ in $\Gamma$ with at most $K|c|$ $\mc{R}$-cells (and we do not need any information
about the number of $\mc{P}$-cells).

These $\mc{R}$-cells lift to a partial filling $\omega$ of $c$ in $\hat{C}$.  
There is 
a collection $\nu_1, \ldots , \nu_k$ of embedded loops, each of which is 
in the closed star of some cone vertex $x_i \in \hat{C}$ so that as (oriented)
$1$-cycles, the boundary of $\omega$ is $c - \sum_i {\nu_i}$.

The edges in the $\nu_i$ are of three types:
\begin{enumerate}
\item edges on the boundary of some $\mc{R}$-cell
\item edges in the Cayley graph whose interior do not intersect any $\mc{R}$-cell; and
\item the cone edges removed when constructing $c'$ from $c$.
\end{enumerate}
There are at most $MK|c|$ edges of the first type, and at most
$|c|$ total edges of the second and third types.  Since the cone on
any graph has isoperimetric constant $1$, each loop $\nu_i$
can be filled with a disk $\Delta_i$ of area at most $|\nu_i|$.  The required filling of $c$ is given by the $2$-chain $\omega+\sum_i{\Delta_i}$.  It is straightforward to see that this can be realized by
a disk.  Since the area of $\omega$ is at most $K|c|$, and 
the sum of the areas of the $\Delta_i$ is equal to 
\[	\sum_i{|\nu_i|} \le (MK + 1)|c|,	\]
we get the required isoperimetric constant for $\hat{C}$.
\end{proof}

\begin{proposition}\label{p:uniformhyp}
Suppose that $G$, $\mc{P}$ and $S$ are as in the hypothesis of Theorem
\ref{t:rhds}, and let $F=\emptyset$.  
There is some $\delta$ and $B$ so that for any hyperbolic filling $G'=G(N_1,\ldots,N_m)$
and $\mc{Q}$ as in Theorem \ref{t:rhds} with $|N_i|_{P_i} \ge B$ for 
all $i$ the space $X(G',\mc{Q},S)$ is
$\delta$-hyperbolic .
\end{proposition}
\begin{proof}
By \cite[Theorem 3.25]{rhds}, $G$ is hyperbolic relative to 
$\mc{P}$.
By the Appendix of \cite{osin:relhypbook} this means that the
relative Dehn function of $G$ with respect to $\mc{P}$ is linear.
Let $C$ be the constant of this linear function.  By \cite[Lemma 5.3]{osin:peripheral}, there is a finite set $\mc{A} \in G$ so that if
each $N_i \cap \mc{A} = \emptyset$ then the relative Dehn function
for $G'$ with respect to $\mc{Q}$ is linear with constant at most $3C$.
Let $B$ be so large that the ball of radius $B$ about $1$ in $G$ contains $\mc{A}$.

Given a finite relative presentation for $G$, there is an obvious
finite relative presentation for $G'$, and the maximum length of 
a relator does not increase.  Let $M$ be the maximum length
of a relator in the given finite relative presentation for $G$ (which is
used to calculate the constant $C$ above).  
By Lemma \ref{l:regulator}, the coned-off Cayley complex for 
$G'$ has
a linear isoperimetric function with constant at most $3(M+1)C +1$. 
Let $C' = 3(M+1)C +1$. 
By \cite[Theorem 3.24]{rhds} this implies that the cusped space for
$G'$ has a linear isoperimetric function with constant 
$3C'(2C'+1)$.  Now, by \cite[Theorem III.H.2.9]{bridhaef:book}, the
constant of hyperbolicity for the cusped space for $G'$ can be calculated
explicitly in terms of this isoperimetric constant, and $\max\{M,5\}$, the maximum
length of an attaching map of a $2$-cell for the cusped space.
Putting all of these estimates together shows that this constant
of hyperbolicity is uniform over all sufficiently long fillings.
\end{proof}

\begin{remark}
There is a direct proof of the above result using the results of 
\cite{rhds} rather than \cite{osin:peripheral}.  However, the output of the main theorem of 
\cite{rhds} is a bound on the constant for a linear {\em homological}
isoperimetric inequality for the space $X(G',\mc{Q},S)$.  In order
to apply this, one needs to make the constant of hyperbolicity in
the conclusion of \cite[Theorem 2.29]{rhds} explicit in terms of 
the homological isoperimetric constant. This would involve rewriting
\cite[Theorem III.H.2.9]{bridhaef:book} in a homological setting.  Feeling
that this would be too much of a diversion, we chose
the shorter but more circuitous proof above.
\end{remark}

\section{Quasi-convexity}
Suppose that  $H$ and $G$ are both relatively hyperbolic
groups.  Let $\mc{P}=\{P_1,\ldots,P_n\}$ be the
peripheral subgroups of $G$, and let $\mc{D}=\{D_1,\ldots,D_m\}$ be
the peripheral 
subgroups of $H$.  Let $\phi\co H\to G$ be a homomorphism.
If every
$\phi(D_i)\in \mc{D}$ is conjugate in $G$ into some $P_j\in \mc{P}$, we say
that the map $\phi$ \emph{respects the peripheral structure on}
$H$.
Let $S$ and $T$ be finite relative generating sets for $G$ and $H$
respectively.

\begin{lemma} \label{l:check}
If $\phi\co H\to G$ is a homomorphism which
respects the peripheral structure on $H$, then
$\phi$ extends to an $H$-equivariant lipschitz map $\check{\phi}$
from (the zero-skeleton of) $X(H,\mc{D},T)$ to
$X(G,\mc{P},S)$.  If $\phi$ is injective, then $\check\phi$ is proper.
\end{lemma}
\begin{proof}
We first associate with each $D_i\in \mc{D}$ an element $c_i \in
G$.
Since $\phi$ respects the peripheral structure, there is some $P_{j_i}\in
\mc{P}$ and some $c\in G$ so that $\phi(D_i) \subseteq c P_{j_i}
c^{-1}$.  We let $c_i$ be some shortest such $c$, with respect to the
generators $S$.

For $h\in H$, we define $\check\phi(h) = \phi(h)$.  A vertex in a
horoball 
of $X(H,\mc{D},T)$ is
determined by a triple $(sD_i,h,n)$, where $s\in H$, $D_i\in \mc{D}$,
$h\in sD_i$, and $n\in \N$.
We define
\[ \check\phi(sD_i,h,n) = (\phi(s)c_iP_{j_i}, \phi(h)c_i, n) .\]
Let $a = \max\{|\phi(t)|_S\mid t\in T\}$, and let $b =
\max\{|c_i|_S\}$; the map $\check\phi$ is $\alpha$-lipschitz for 
$\alpha = \max\{a, b+1\}$.

Properness is left to the reader.
\end{proof}

Recall that a  \emph{filling} of $G$ is determined by a choice of subgroups
$N_j\lhd P_j$, called \emph{filling kernels}; we write the quotient
after filling as $G(N_1,\ldots,N_m)$.

\begin{definition} \label{d:inducedfilling}
If $\phi$ is a homomorphism which respects the peripheral
structure on $H$, then any filling of $G$ induces a filling of $H$ as
follows.  For each $i$, there is some $c_i=c(D_i)$ in $G$ and some
$P_{j_i}$ in $\mc{P}$
so that
$c_i P_{j_i} c_i^{-1}$ contains $\phi(D_i)$.  
The induced filling kernels $K_i\lhd D_i$ are given by
\[ K_i = \phi^{-1}(c_i N_{j_i} c_i^{-1})\cap D_i.\]
The induced filling is $H(K_1,\ldots,K_n)$.
The map $\phi$ induces a homomorphism
\[\bar{\phi}\co H(K_1,\ldots,K_n)\to G(N_1,\ldots,N_m).\]
\end{definition}

\begin{definition}
Suppose $G$ is a relatively hyperbolic group, relative to $\mc{P}$,
and that $H<G$ is hyperbolic relative to $\mc{D}$ and that the
inclusion of $H$ into $G$ respects the peripheral structure.
A filling
$G\to G(N_1,\ldots,N_m)$ is an \emph{$H$-filling} if whenever 
$H\cap P_i^g$ is nontrivial, $N_i^g\subseteq sD_js^{-1}\subseteq H$
for some $s\in H$, and $D_j\in \mc{D}$.
\end{definition}

\subsection{Induced peripheral structures}

Let $G$ be a hyperbolic group, and let $H<G$ be a quasi-convex
subgroup. 
Recall that according to Theorem \ref{t:finiteheight}, $H$ has finite
\emph{height} (see  Definition \ref{d:height}).  We will construct a
peripheral structure for $G$ using the infinite intersections of
maximal collections of essentially distinct conjugates of $H$.

\begin{lemma} There are only finitely many $H$-conjugacy classes of
  subgroups $H\cap H^g$ such that $|H\cap H^g|=\infty$. 
\end{lemma}
\begin{proof}
Two double coset representatives $g_1, g_2$ of $HgH$
give the same $H$-conjugacy class $H\cap H^{g_i}< H$. By \cite[Lemma 1.2]{gmrs}, there is an upper
bound on the minimal length of a coset representative of $HgH$ such that $|H\cap H^g|=\infty$. 
\end{proof}

Using induction on the height we obtain the following.
\begin{corollary}
There are only finitely many $H$-conjugacy classes of 
intersections $H\cap H^{g_2}\cap \cdots\cap H^{g_j}$, where $j\leq n$,
with $n$ the height of $H$ in $G$ and $\{ 1, g_2, \ldots, g_j\}$  are essentially distinct. 
\end{corollary}

Choosing one subgroup of this form per
$H$-conjugacy class and taking its commensurator in $H$, 
we obtain a system $\mc{D}$ of (quasi-convex) subgroups of $H$
which we will call the \emph{malnormal core of $H$}.  
The collection $\mc{D}$ gives rise to a collection of peripheral
subgroups $\mc{P}$ for $G$ in two steps:  
\begin{enumerate}
\item Change $\mc{D}$ to $\mc{D}'$ by replacing each element of
  $\mc{D}$ by its commensurator in $G$.
\item Eliminate redundant entries of $\mc{D}'$ to obtain
  $\mc{P}\subseteq\mc{D}'$ which contains no two elements which are
  conjugate in $G$.
\end{enumerate}
Call $\mc{P}$ the peripheral structure on $G$ \emph{induced} by $H$.
This peripheral structure is only well-defined up to
replacement of some elements of $\mc{P}$ by conjugates.  
On the other hand, replacing $H$ by a commensurable subgroup of $G$
does not affect the induced peripheral structure.  
We consider two peripheral structures on a group to be the same if the
same group elements are parabolic in the two structures.
\begin{observation}
Let $H_1$ and $H_2$ be quasi-convex subgroups of the hyperbolic group 
$G$ with the same limit sets in $\partial G$.  
The peripheral structures induced by $H_1$ and $H_2$ are the same.
\end{observation}

In the next two observations and lemma, we consider a hyperbolic group
$G$ and a quasi-convex subgroup $H$.
We let $\mc{D}$ be the malnormal core of $H$, and let $\mc{P}$ be the
peripheral structure on $G$ induced by $H$.  Finally, 
\[\check\iota\co X(H,\mc{D},T)\to X(G,\mc{P},S) \]
is the extension of the inclusion map given by Lemma \ref{l:check}.
\begin{observation}\label{o:induced}
If $\mc{P}$ is the peripheral
structure on $G$ induced by $H$, and $P\in \mc{P}$, then $H\cap P$ is
finite index in $P$. 
\end{observation}
\begin{observation}\label{o:noaccidents}
If $\mc{P}$ is the peripheral structure on $G$ induced by $H$, and
$h\in H$ is parabolic with respect to that structure, then $h$ is
conjugate in $H$ to an element of $D$ for some $D$ in the malnormal
core $\mc{D}$ of $H$.
\end{observation}
\begin{lemma} \label{l:allornothing}
There is a constant $\beta$
satisfying the following:  Let $A$ be a horoball of $X(G,\mc{P},S)$,
and suppose that $H\cap \stab_G(A)$ contains an element of infinite order.  A
$\beta$-neighborhood of the image of $\check\iota$ contains $A$.
\end{lemma}
\begin{proof}
For each $D \in \mc{D}$, there is some (unique) $P=P(D)\in \mc{P}$ and some $c=c(D)\in G$
(chosen as in the proof of Lemma \ref{l:check}) so that 
$D < cPc^{-1}$.  By Observation \ref{o:induced}, $D$ is finite
index in $c Pc^{-1}$.  Since $\mc{D}$ is finite, there is some
constant $\beta_1$, independent of $D$, so that $cP$ is contained in a
$\beta_1$-neighborhood of $D$ in $G$.

Let $h$ be the infinite order element of $H\cap \stab_G(A)$.
Observation \ref{o:noaccidents} implies that $h$ is already parabolic
in $H$, so $h\in sDs^{-1}$ for some $s\in H$ and $D\in \mc{D}$.

The horoball $A$ is attached to some coset $tP_i$ for $t\in G$ and
$P_i\in \mc{P}$. 
Since parabolics cannot have
infinite intersection without coinciding, it follows that $P_i =
P(D)$; we may take $t=sc$.  

It follows from the first paragraph that $tP_i=scP_i$ is contained
in a $\beta_1$-neigh\-bor\-hood of $sD$.  Moreover, elements of $tP_i$ are
uniformly close to elements of $sDc$, and elements of the horoball $A$
are uniformly close to elements of the form
\[\check\iota(sD, h, n) = (t P_i, hc, n). \]
In other words, the vertices of $A$ which do not lie in $G$ are all
contained in some $\beta_2$-neighborhood of the image of
$\check\iota$.

We may therefore take $\beta = \max\{\beta_1,\beta_2\}$.
\end{proof}

\begin{example}
Let $G = \langle a, b\rangle$ be a free group of rank $2$ and let 
$H = \langle a^2 , b a^3 b^{-1}\rangle$.  In this case one must take
commensurators twice, once to get the malnormal core and a second
time to get the induced peripheral structure.  Indeed, the 
intersections of $H$ with its conjugates are conjugate in $H$ either to
\[ \langle a^6\rangle = H\cap H^a\cap H^b\cap H^{ab^{-1}}\cap H^{a^2b^{-1}}\]
 or 
\[\langle ba^6b^{-1}\rangle = H\cap H^{bab^{-1}}\cap
H^{ba^2b^{-1}}\cap H^b\cap H^{ba},\] 
so $H$ has height $5$.
The malnormal core of $H$ is 
\[\mathcal{D} = \{\langle a^2\rangle, \langle b a^3 b^{-1}\rangle\}\]
and the induced peripheral structure on $G$ is 
$\mc{P} = \{\langle a \rangle\}$.
\end{example}

\begin{definition} 
Let $\phi\co H\to G$ be a homomorphism which respects the peripheral
structure.  We say that $\phi(H)$ is $C$-\emph{relatively quasi-convex} in $G$
if $\check{\phi}$ has $C$-quasi-convex image.  If $H<G$ and $\phi$ is the
inclusion map, we say that $H$ is a relatively quasi-convex
subgroup of $G$.
\end{definition}
The relative quasi-convexity of $\phi(H)$ does not depend on the
choice of relative generating sets $S$ and $T$, though the constant
$C$ does depend on $S$ and $T$.

\begin{proposition}\label{p:induced}
Let $H$ be a quasi-convex subgroup of the torsion-free hyperbolic group $G$, and
let $\mc{D}$ be the malnormal core of $H$.  
Let $\mc{P}$ be the peripheral structure on $G$ induced by $H$.
\begin{enumerate}
\item $H$ is hyperbolic relative to $\mc{D}$.
\item $G$ is hyperbolic relative to $\mc{P}$.
\item\label{pt3} With respect to the above peripheral structures, $H$ is a
  relatively quasi-convex subgroup of $G$.
\end{enumerate}
\end{proposition}
\begin{proof}
The first two assertions are essentially contained in \cite[Proposition 7.11]{bowditch:relhyp}.  By construction, the elements of $\mc{D}$ are quasi-convex,
non-conjugate, and any pair of conjugates of elements of $\mc{D}$ are either
equal or intersect in a finite set.  They are also equal to their commensurator, are
hence to their normalizer.  These are the hypotheses of \cite[Proposition 7.11]{bowditch:relhyp}.  The same properties hold for $\mc{P}$ in $G$.

We now consider the third property.  Let $X = X(G,\mc{P},S)$ be the cusped space 
of $G$. Let $X_H$ be the zero-skeleton of the cusped space of $H$, and let $Y$ be the image
of the proper map $\check\phi\co X_H\to X$ from Lemma \ref{l:check}.
Let $x$ and $y$ be vertices of $Y$.

We need to prove that there is a constant $C$ (independent of $x$ and $y$) so that 
a geodesic in $X$ between $x$ and $y$ lies within $C$ of $Y$.

\begin{case}
The points $x$ and $y$ lie deep (deeper
than $50\delta$) in the same
horoball.  
\end{case}
In this case the geodesic between $x$ and $y$ lies entirely in the
horoball (see \cite[Lemma 3.26]{rhds}).
Any geodesic between $x$ and $y$ is Hausdorff distance at most $4$
from a regular geodesic, which is vertical except for a horizontal
segment of length at most three (see \cite[Lemma 3.10]{rhds}).
The vertical sub-segments start at points in $Y$, so by construction of
cusped spaces and the map 
$\check\phi$, the vertical sub-segments lie in $Y$ also.  Therefore in this case we can
take $C = 6$.

\begin{case}
$x$ and $y$ lie at depth no more than $50\delta$ in $X$.
\end{case}
In this case, consider the space $X'$ which consists of all vertices in $X$ at depth
at most $100\delta$.  This space is quasi-isometric to the group $G$, and $H$ is a
quasi-convex subset of $X'$, with quasi-convexity constant $\lambda$, say.  
Choose a geodesic $\gamma$  in $X'$ between $x$ and $y$.  We may
assume that $\gamma$ is ``regular'' in each horoball, in the following
sense:  If $\gamma$ contains vertices at depth $90\delta$ in the
horoball, then that part of $\gamma$ between depth $0$ and depth
$90\delta$ consists of two vertical segments.

Since $H$ is $\lambda$-quasi-convex in $X'$,
 there is an element of $H$ within $\lambda$ of any point in $\gamma$.  Now
consider $\gamma$ as a subset of $X$, using the natural inclusion of $X'$ in $X$.
We will replace $\gamma$ with a $10\delta$-local geodesic
$\bar{\gamma}$ in $X$ with endpoints
$x$ and $y$.  The path $\bar{\gamma}$ will be seen to lie in a
uniformly bounded neighborhood of $Y$.

Let $\sigma$ be a sub-segment of $\gamma$ lying entirely below depth
$90\delta$.  Any such $\sigma$ is contained in a unique maximal
segment $\hat{\sigma}$ lying below depth $80\delta$.  To define
$\bar{\gamma}$, we replace each
such $\hat{\sigma}$ with an $X$-geodesic consisting of two vertical
and one horizontal sub-segment.

This yields a continuous path $\bar{\gamma}$ from
$x$ to $y$ which we claim is a $10\delta$-local geodesic in $X$.  
Consider a sub-segment $I$ of $\bar{\gamma}$ of length $10\delta$.
We must show that $I$ is a geodesic in $X$.
If $I$ lies completely beneath 
depth $80\delta$ it is obvious that $I$ is
geodesic.  

Suppose $I$ lies entirely above depth $80\delta$.  Any path joining
the endpoints of $I$ which is not entirely contained in $X'$ must have
length at least $40\delta$.  Since $I$ has length $10\delta$ and is a
geodesic in $X'$, $I$ is a geodesic in $X$.

Finally, between depths $70\delta$ and $90\delta$, $\bar{\gamma}$ is
vertical, and hence geodesic.  In particular, if $I$ crosses depth
$80\delta$, then $I$ is geodesic.
  This shows that $\bar{\gamma}$ is a $10\delta$-local
geodesic between $x$ and $y$.

We claim that $\bar{\gamma}$ lies in a bounded neighborhood of $Y$.
This is clear for those parts of $\bar{\gamma}$ which lie in $\gamma$, so
let $\tau$ be a maximal sub-segment of $\bar{\gamma}$ which is not contained
in $\gamma$.  Then $\tau$ is contained in a single horoball $A$.
We now split into two subcases, depending on the length of $\tau$.  Let $A_0$ be
the part of $A$ at depth $0$, and $A_\lambda$ be the $\lambda$-neighborhood of
$A_0$ in $X$.  Also let $G_A = \text{Stab}(A)$.  Then $G_A$ acts 
cocompactly on $A_\lambda$. Let $K$ be the number of vertices in $A_\lambda / G_A$.

\begin{subcase}
The length of $\tau$ is less than $\Const$.
\end{subcase} 
Then each point in $\tau$ is at most $\Const/2$ from an endpoint of $\tau$.  However,
the endpoints of $\tau$ lie in $\gamma$, which lies in a $\lambda$-neighborhood of
$Y$.  Thus in this case each element of $\tau$ lies within $\Const/2 + \lambda$ of $Y$.

\begin{subcase}
The length of $\tau$ is at least $\Const$.
\end{subcase}
In this case, consider the path $\tilde{\tau} \subset \gamma$ which joins the
endpoints of $\tau$.  The path $\tilde{\tau}$ has length at least $\logConst$.  
Each point in $\tilde{\tau}$ is within at most $\lambda$ from a point in $H$, so there
are at least $K+1$ distinct points, $\{ h_1, \ldots , h_{K+1} \}$, all within $\lambda$ of $\tilde{\tau}$ and each of these 
points lies in $A_\lambda$.  By the choice of $K$, there is $h_i \ne h_j$ in the same
$G_A$-orbit, so $h_ih_j^{-1} \in G_A\smallsetminus\{1\}$.  
Since $G$ is torsion-free, $h_ih_j^{-1}$ has infinite order, and
by Lemma \ref{l:allornothing}, a $\beta$-neighborhood of $Y$ contains
$A$.

Let $\eta$ be a geodesic joining $x$ to $y$ in $X$.
By \cite[III.H.1.13(1)]{bridhaef:book}, $\eta$ lies in a $2\delta$
neighborhood of $\bar{\gamma}$, which we have already shown lies in a bounded
neighborhood of $Y$.

\begin{case}
Either $x$ or $y$ lies inside a
$50\delta$-horoball, but we are not in Case 1.   
\end{case}
If $x$ or $y$ lies in a horoball, it lies
directly beneath a point in $H$ at depth $0$ (i.e. in the Cayley graph
of $G$) in $X$.  Either appending or deleting
\footnote{Whether a vertical path is appended or deleted depends on whether the
geodesic from $x$ to $y$ initially goes up or down (the case when it goes horizontal is
treated as if it goes down).}
the vertical paths from $x$ to depth $0$ and similarly for $y$ we obtain a path which
is a $10\delta$-local geodesic.  Since $10\delta$-local geodesics are
$(\frac{7}{3},2\delta)$-quasi-geodesics
\cite[III.H.1.13(3)]{bridhaef:book}, the proposition follows now from
Cases 1 and 2.
\end{proof}

\begin{remark}
We direct the interested reader to \cite{mp} for a much more general
theorem from which part \eqref{pt3} of Proposition \ref{p:induced} follows.
\end{remark}

In general, even if $G$ is hyperbolic, a relatively quasi-convex subgroup
(with respect to some relatively hyperbolic structure on $G$) need
not be quasi-convex in $G$.  However, the following lemma is
straightforward.

\begin{lemma} \label{l:finiteparabolics}
Suppose that $G$ is hyperbolic relative to a collection of finite subgroups.  Then $G$ is hyperbolic and any relatively quasi-convex
subgroup of $G$ is quasi-convex.
\end{lemma}

\section{Proof of Theorem \ref{t:technical}}

\subsection{Projections of geodesics to cusped spaces of quotients}
The key technical lemma is the following:
\begin{lemma}\label{l:shortcutfailure}
Fix a relatively hyperbolic group $G$ with peripheral subgroups
$\mc{P}$ and compatible generating set $S$.  
Choose $\delta>0$ so that
the cusped spaces $X=X(G,\mc{P},S)$ and
$X'=X(G(N_1,\ldots,N_m),\mc{Q},S)$ are both $\delta$-hyperbolic,
whenever $G(N_1,\ldots,N_m)$ is a sufficiently long filling of $G$.

Let $L\geq 10\delta$, and let $D\geq 3 L$.
Let $F=\{g\in G\mid d_X(g,1)\leq 2D\}$, and
let $G\stackrel{\pi}{\longrightarrow}G(N_1,\ldots,N_m)$ be any
hyperbolic filling of $G$ which is injective on $F$ and so that
$X'=X(G(N_1,\ldots,N_m),\mc{Q},S)$ is $\delta$-hyperbolic.
(We denote the induced map from $X$ to $X'$ also by $\pi$.)
Let $\gamma$ be a regular
geodesic in $X$ joining two elements of $G$.  
One of the following occurs:
\begin{enumerate}
\item\label{shortcuttable} There is a $10\delta$-local geodesic with the same endpoints as
  $\pi(\gamma)$ which is contained in a $2$-neighborhood of
  $\pi(\gamma)$ and coincides with $\pi(\gamma)$ everywhere in an
  $L$-neighborhood of the Cayley graph of $G(N_1,\ldots,N_m)$.
\item\label{notshortcuttable} There is a coset $tP_i$ whose corresponding horoball intersects
  $\gamma$ in a sub-segment $[g_1,g_2]$ of length at least $2D-20\delta-4$
  but there is some $n\in N_i$ with $d_X(g_1,g_2n)\leq 2L+3$.
\end{enumerate}
\end{lemma}
\begin{proof}
It is straightforward to verify that the map $\pi\co X\to X'$ induced
by the filling is injective on balls 
of radius $10\delta$ centered on points either in the Cayley graph or 
at depth at most $D-10\delta-2$ in a horoball.  It follows that
sub-segments of $\gamma$ straying no further than $D-10\delta-2$ from the Cayley
graph project to $10\delta$-local geodesics.  Let $B$ be a
horoball which $\gamma$ penetrates to depth greater than $D-10\delta-2$.  
The
horoball $B$ intersects the Cayley graph of $G$ in some coset $tP$
for $t\in G$ and $P_i\in \mc{P}$.  Let $g_1$ and $g_2$ be the initial
and terminal vertices of $\gamma\cap B$.  There are three
possibilities:
\begin{enumerate}[(a)]
\item\label{geodesic} $\pi(\gamma\cap B)$ is geodesic,
\item\label{shortcut} $\pi(\gamma\cap B)$ is not geodesic but
$d_{X'}(\pi(g_1),\pi(g_2))\geq 2L+3$,
\item\label{contra} $\pi(\gamma\cap B)$ is not geodesic and
$d_{X'}(\pi(g_1),\pi(g_2))< 2L+3$.
\end{enumerate}
We first claim that if \eqref{geodesic} or \eqref{shortcut} holds for
every horoball which $\gamma$ penetrates to depth
greater than $D-10\delta-2$, then 
conclusion \eqref{shortcuttable} of the Lemma holds.  We argue by
constructing
a new $10\delta$-quasi-geodesic 
$\gamma'$ in $X'$ which agrees with $\pi(\gamma)$
everywhere in an $L$-neighborhood of the Cayley graph of
$G(N_1,\ldots,N_m)$ and inside those horoballs of $X'$ which
$\pi(\gamma)$ intersects in a geodesic segment.  Whenever a sub-segment
$\sigma$
of $\pi(\gamma)$ is of type \eqref{shortcut}, we can replace it by a
shorter, but still $10\delta$-local geodesic segment as follows.
The segment $\sigma$ is composed of two vertical sub-segments and a
short (length $2$ or $3$) horizontal sub-segment at depth $d > D-10\delta-2$.
Since $\pi(\sigma)$ is not geodesic, the images in $X'$ of the vertical
sub-segments must come within a horizontal distance of $3$ of one
another at some smaller depth $d'$.
The assumption that $d_{X'}(\pi(g_1),\pi(g_2))\geq 2L+3$
forces $d'>L$.  Modifying $\sigma$ by removing the part lying below depth
$d'$ and replacing it with a horizontal geodesic leaves a geodesic
$\sigma'$ which still goes to depth at least $L$.  Therefore
making all possible such modifications leaves a $10\delta$-local
geodesic $\gamma'$ satisfying conclusion \eqref{shortcuttable} of the
Lemma.

Now suppose that there is some horoball $B$ so that $\gamma$
penetrates $B$ to depth greater than $D-10\delta-2$, but $\pi(\gamma\cap B)$
satisfies condition \eqref{contra} above.  The image of $P_i$ in
$G(N_1,\ldots,N_m)$ is canonically isomorphic to $P_i/N_i$, so there is
some $n\in N_i$ so that 
$d_X(g_1,g_2n)\leq 2 L+3$.
Since $\gamma$ is geodesic, $d_X(g_1,g_2) > 2(D-10\delta-2)$, and so conclusion
\eqref{notshortcuttable} holds.
\end{proof}

In our current applications, we will only ever apply this lemma to a
geodesic with both endpoints in a quasi-convex subgroup.
In this context, more can be said.
\begin{lemma} \label{l:technical2}
Let $G$, $\mc{P}$, $S$, and $L\geq 10\delta$
be as in the
hypothesis of Lemma \ref{l:shortcutfailure}.  Let $H$ be a
$\lambda$-relatively quasi-convex subgroup of $G$, and let $\alpha$
be the lipschitz constant for the extension of the inclusion map in
Lemma \ref{l:check}.
 
Let $D\geq 3L+100\lambda+4\alpha$, and let $F=\{g\in G\mid d_X(g,1)\leq 2D\}$.
Suppose
$G\stackrel{\pi}{\longrightarrow}G(N_1,\ldots,N_m)$ is an $H$-filling
which is injective on $F$ and so that
$X'=X(G(N_1,\ldots,N_m),\mc{Q},S)$ is $\delta$-hyperbolic.  Let
$K_H<\ker(\pi)\cap H$
be the kernel of the induced filling of $H$.
Finally, suppose that $\gamma$ is a geodesic joining $1$ to $h$ for
some $h\in H$. 

If conclusion \eqref{notshortcuttable} of Lemma
\ref{l:shortcutfailure} holds, then there is some $k\in K_H$
satisfying $|kh|_X<|h|_X$.\footnote{(writing $|\cdot|_X$ for $d_X(\cdot,1)$)}
\end{lemma}
\begin{proof}
Let $g_1,g_2\in tP_i$ and $n\in N_i$ be as
in the conclusion to Lemma \ref{l:shortcutfailure},  and let $B$ be the
horoball in $X$ which contains $tP_i$.
We have
\begin{eqnarray}
\label{deep} d_X(g_1,g_2)& \geq & 2D-20\delta-4 \ge 
6L+200\lambda+8\alpha-20\delta-4\mbox{, but}\\
\label{shallow} d_X(g_1,g_2n) & < & 2L+3.
\end{eqnarray}

We use quasi-convexity to approximate $g_1$ and $g_2$ by elements of
$Hc\cap tP_i$ for some small $c\in G$.  If $g\in tP_i$ and $m\in \N$, we will write $(g,m)$ for the
unique vertex of $B$ connected to $g$ by a vertical geodesic of length
$m$, as in Remark \ref{r:orderedpair}.

By \eqref{deep}, the geodesic $\gamma$ penetrates the
horoball $B$ to depth at least
$D-10\delta - 4 > 2\lambda$; 
in particular, $\gamma$ passes through
$(g_1,\lambda+1)$ and $(g_2,\lambda+1)$. Let $\iota\co H\to G$ be the
inclusion map, and $\check\iota$ the extension from Lemma \ref{l:check}.
Since $H$ is
$\lambda$-relatively quasi-convex, there are points 
$z_1=\check\iota\left((s_1 D_{j_1},h_1,n_1)\right)$ and
$z_2=\check\iota\left((s_2 D_{j_2},h_2,n_2)\right)$ in $B$ within $\lambda$ of
$(g_1,\lambda+1)$ and $(g_2,\lambda+1)$, respectively.  
Note that $d_X(z_j,h_j)\leq 2\lambda+\alpha$  and thus
$d_X(h_j,g_j)\leq 4\lambda+\alpha$ for $j=1,2$.
Thus
\begin{equation}\label{hjgj}
 d_X(h_1,h_2)\geq d_X(g_1,g_2) -( 8\lambda+2\alpha) > 0.
\end{equation}
In particular $h_1\neq h_2$.
Moreover, since
$\check\iota$ is $H$-equivariant, $h_2h_1^{-1} B$ intersects $B$ in
its interior.  Hence $h_2h_1^{-1}$ fixes the horoball $B$ and the coset
$tP_i$.  
Because the filling kernels
$\{N_1,\ldots,N_m\}$ are assumed to determine an $H$-filling, we have
\[ tN_it^{-1}\subseteq sD_l s^{-1}\subseteq H\]
for some $s\in H$ and some $D_l\in \mc{D}$.

For $j\in \{1,2\}$ we have $g_j = t p_j$ for $p_j\in P_i$.
As $N_i\lhd P_i$, 
$g_2n = tp_2 n = t n' p_2$ for some $n'\in N_i$.  
Let $k = tn't^{-1}$, so that $g_2n = kg_2$.

We claim that $k\in K_H$.  Indeed, $K_l = c N_i c^{-1}$ is the induced
filling kernel in $D_l$, if $c = c(D_l)$.  Moreover $tN_it^{-1} =
scN_ic^{-1}s^{-1}=sK_ls^{-1}$ lies in $K_H$.  In particular $k\in
sK_ls^{-1}\subset K_H$.

Let $h' = kh$.  It remains to show
that $|h'|_X<|h|_X$.

Note first that $d_X(h_2',g_2 n)=d_X(h_2,g_2)$.
It follows that
\begin{equation}\label{h1h2n}
|h_1^{-1}h_2'|_X < 8\lambda+2\alpha+2L+3.
\end{equation}
Clearly $h = h_1 (h_1^{-1}h_2)(h_2^{-1}h)$.
Furthermore, each of
$h_1$ and $h_2$ lies no more than $4\lambda+\alpha$ from a geodesic connecting
$1$ to $h$.  Thus
\begin{equation} \label{exceed}
|h_1|_X+|h_1^{-1}h_2|_X+|h_2^{-1}h|_X \le |h|_X + 16\lambda+4\alpha.
\end{equation}

We can factorize $h'$ as $h' = h_1 (h_1^{-1}h_2') (h_2^{-1}h)$.  
It follows that
\begin{eqnarray}
\nonumber |h'|_X &\le& |h_1|_X + |h_1^{-1}h_2'|_X + |h_2^{-1}h| \\
\label{diff}&< &  |h|_X+16\lambda+4\alpha - (|h_1^{-1}h_2|_X-|h_1^{-1}h_2'|_X),
\end{eqnarray}
by \eqref{exceed}.

Inequalities \eqref{hjgj}, \eqref{h1h2n}, and \eqref{deep} imply that
\begin{eqnarray*}
|h_1^{-1}h_2|_X-|h_1^{-1}h_2'|_X & > &
|g_1^{-1}g_2|_X - (8\lambda+2\alpha) - (8\lambda+2\alpha+2L+3)\\
& = & |g_1^{-1}g_2|_X-(16\lambda+4\alpha+2L+3)\\ 
& \geq & 200\lambda+2L + 20\delta-4 -(16\lambda+4\alpha+2L+3)\\
& > & 16\lambda.
\end{eqnarray*}
Applied to \eqref{diff},
this yields $|h'|_X< |h|_X$, as required.
\end{proof}

\subsection{The image of $H$ is quasi-convex}

\begin{proposition}\label{p:fillqc}
Let $G$ be a relatively hyperbolic group and $H$ a subgroup
as in Definition \ref{d:inducedfilling}, and suppose that $H$ is
$\lambda$-relatively quasi-convex in $G$.
For all sufficiently
long $H$-fillings $G(N_1,\ldots,N_m)$ of $G$,
the image in
$G(N_1,\ldots,N_m)$ of the induced filling
$H(K_1,\ldots,K_n)$ is $\lambda'$-relatively quasi-convex, for some
$\lambda'$ independent of the filling.
\end{proposition}
\begin{proof}
We fix a (relative, compatible) generating set $S$ for $G$, and 
let $\delta$ be the uniform constant of hyperbolicity for cusped spaces
$X=X(G,\mc{P},S)$ and $X'=X(G(N_1,\ldots,N_m),\mc{Q},S)$ provided by
Proposition \ref{p:uniformhyp}.  It is useful to assume that both
$\delta$ and $\lambda$ are integers.

We will apply Lemmas \ref{l:shortcutfailure} and \ref{l:technical2} with
$L = 10\delta$ and $D = 100\lambda + 100\delta$.  
``Sufficiently long'' then means that the filling is injective on 
$F = \{ g \in G \mid d_X(g,1) \le 2D \}$ (and that $X'$ is $\delta$-hyperbolic).

By \cite[III.H.1.13]{bridhaef:book}, any $10\delta$-local geodesic in
$X'$ is a $(7/3,2\delta)$-quasi-geodesic.
Let $R$ be the constant of quasi-geodesic stability for
$(7/3,2\delta)$-quasi-geodesics in a $\delta$-hyperbolic space (see
\cite[III.H.1.7]{bridhaef:book}). 
We show that it is sufficient to take $\lambda'=\lambda+R+2\delta+2$.  

Let $\iota \co H \to G$ be inclusion and $\phi \co H \to G(N_1,\ldots, N_m)$ be $\pi \circ \iota$ where $\pi \co G \to G(N_1,\ldots,N_m)$
is the filling map.  Recall from Lemma \ref{l:check} that we have induced
maps $\check{\iota}$ and $\check{\phi}$ from the cusped space for
$H$ to $X$ and $X'$, respectively.

\begin{claim}\label{c:simpler}
Let $\bar{h}\in \pi(H)$.  Any geodesic in $X'$ joining $1$ to
$\bar{h}$ stays in an $(\lambda +R+2)$-neighborhood of the image of
$\check{\phi}$.
\end{claim}
\begin{proof}(Claim \ref{c:simpler})
We choose $h\in H$ of minimal $X$-length projecting to $\bar{h}$, and
let $\gamma$ be a regular geodesic joining $1$ to $h$ in $X$.

By Lemmas \ref{l:shortcutfailure} and \ref{l:technical2}, and the
minimality of $h$, there is a $10\delta$-local geodesic with endpoints
$1 = \pi(1)$ and $\bar{h} = \pi(h)$ which is contained in a 
$2$-neighborhood of $\pi(\gamma)$.  

Any geodesic from $1$ to $\pi(h)$
therefore lies in a $(R+2)$-neighborhood of $\pi(\gamma)$, by
quasi-geodesic stability. Since $\gamma$ lies in a
$\lambda$-neighborhood of the image of $\check{\iota}$, any
geodesic
from $1$ to $\pi(h)$ lies in a $(\lambda+R +2)$-neighborhood of the image
of $\check{\phi}$. 
\end{proof}

Claim \ref{c:simpler} suffices to prove the Proposition, as follows:
Let $X_H$ be the zero-skeleton of the cusped space of $H$, and let $Y=
\check{\phi}(X_H)$.  Let $x_1$, $x_2$ be elements of $Y$.  

If $x_1$ and $x_2$ lie in the same horoball one may use the convexity of
$\delta$-horoballs (\cite[Lemma 3.26]{rhds}) to see that any geodesic
joining them stays in a $(2\delta+2)$-neighborhood of $Y$.  

Suppose
therefore that $x_1$ and $x_2$ lie in different horoballs.  Each $x_i$
is connected by a vertical geodesic to some $h_ic$ for
$h_i\in \phi(H)$ and $|c|_X<\alpha$, where $\alpha$ is the lipschitz constant
from Lemma \ref{l:check}. Except for $h_ic$ itself, this
vertical geodesic contains only vertices of $Y$.  
The geodesic between
$h_1$ and $h_2$ is a $\phi(H)$-translate of one between $1$ and
$h_1^{-1}h_2$, and so this geodesic stays in a
$(\lambda+R+2)$-neighborhood of $Y$ by Claim \ref{c:simpler}.  The two vertical segments, the
geodesic between $h_1$ and $h_2$, and the geodesics from $h_1$ to
$h_1c$ and from $h_2$ to $h_2c$
 form five sides of a geodesic
hexagon, the sixth side of which can be taken to be any geodesic joining $x_1$ to
$x_2$.  This sixth side stays within a $4\delta$-neighborhood of the
other five, and therefore within $\lambda+R+2+4\delta+\alpha$ of $Y$.

The Proposition is proved, for $\lambda' = \lambda+R+2+4\delta+\alpha$.
\end{proof}

The following result is not required for the proof of Theorem \ref{t:technical}, 
but may be of independent interest.

\begin{proposition}\label{p:fillinject}
Let $H<G$ be a relatively quasi-convex subgroup.  For any sufficiently
large $H$-filling $G(N_1,\ldots,N_m)$ of $G$, the induced map from the
induced filling $H(K_1,\ldots,K_n)$ into $G(N_1,\ldots,N_m)$ is
injective.
\end{proposition}
\begin{proof}
As above, choose a compatible
generating set $S$ for $G$ with peripheral structure $\mc{P} =
\{P_1,\ldots,P_m\}$ and let $\delta$ be a constant of hyperbolicity
which suffices both for $X(G,\mc{P},S)$ and for the cusped space $X'$ 
of any sufficiently long hyperbolic filling of $G$.  Let $\lambda$ be the
constant of (relative) quasi-convexity for $H$.  

Once again, we will apply Lemmas \ref{l:shortcutfailure} and \ref{l:technical2} with
$L = 10\delta$ and $D = 100\lambda + 100\delta$.  
``Sufficiently long'' again means that the filling is injective on 
$F = \{ g \in G \mid d_X(g,1) \le 2D \}$ (and that $X'$ is $\delta$-hyperbolic).

Let $\pi \co G \to G(N_1,\ldots,N_m)$ be such a filling.
Let $h \in \ker(\pi) \cap H$ be nontrivial.  We must show that $h \in K_H$, the kernel of the
induced filling of $H$. Let $\gamma$ be a geodesic in $X$ from $1$ to $h$.  Note
that $\pi(\gamma)$ is a loop.  Suppose
that conclusion \eqref{shortcuttable} of Lemma \ref{l:shortcutfailure} holds.  Then there
is a $10\delta$-local geodesic loop based at $1$ in $X'$, which coincides with $\pi(\gamma)$ on an initial segment of length $L \ge 10\delta$.  This is impossible
since there are no nontrivial $10\delta$-local geodesic loops
in a $\delta$-hyperbolic space.  

Therefore we may apply Lemma \ref{l:technical2} to conclude that there is a $k \in K_H$ so that $|kh|_X < |h|_X$.  Induction on the length of $h$ shows that $h \in K_H$,
as required.
\end{proof}

\subsection{Keeping $g$ out of $H$}

\begin{proposition}\label{p:nocollide}
Let $H<G$ be a relatively quasi-convex subgroup, and let
$I>0$.  There is some $F=F(I)$ so that if $G(N_1,\ldots,N_m)$ is an 
$H$-filling of $G$ which is injective on $F$, and $g\in
G\smallsetminus H$ satisfies $|g|_X<I$, then $\pi(g)\notin \pi(H)$.
\end{proposition}
\begin{proof}
  As above, choose a compatible
generating set $S$ for $G$ with peripheral structure $\mc{P} =
\{P_1,\ldots,P_m\}$ and let $\delta$ be a constant of hyperbolicity
which suffices both for $X(G,\mc{P},S)$ and for the cusped space of
any sufficiently long hyperbolic filling of $G$.  Let $\lambda$ be the
constant of 
(relative) quasi-convexity for $H$.

As usual, we will apply Lemmas \ref{l:shortcutfailure} and \ref{l:technical2}.  This time, we will choose $L = 2 I +10\delta$ and $D = 100\lambda+100\delta+6I$.  
Let
\[F = \{g\in G\mid |g|_X\leq 2D\},\]
and let $G(N_1,\ldots , N_m)$ be an $H$-filling of $G$ which is injective
on $F$ and so that the associated cusped space $X'$ is $\delta$-hyperbolic.

If the proposition does not hold,
then there is some $g\in G\smallsetminus H$ so $|g|_X<I$, and some
$h\in H$ so that $\pi(g)=\pi(h)$.  Without loss of generality, we may
pick some such $h$ so $|h|_X$ is minimal.  Note that $|h|_X\geq
200\lambda+200\delta+12I$, by the injectivity hypothesis.
We let $\gamma$ be a
geodesic joining $1$ to $h$.  By Lemmas \ref{l:shortcutfailure} and
\ref{l:technical2}, and the minimality of $h$, conclusion \eqref{shortcuttable} of Lemma \ref{l:shortcutfailure} holds.

In this
case, there is a $10\delta$-local geodesic $\gamma'$ in $X'$ connecting $1$ to
$\pi(g)$ which lies in a $2$-neighborhood of $\pi(\gamma)$ and
coincides with $\pi(\gamma)$ in a $(2I+10\delta)$-neighborhood of both $1$
and $\pi(g)$.  It follows that $\gamma'$ has length at least
$4I+20\delta$.  But since $\gamma'$ is a $10\delta$-local
geodesic, it must be a $(7/3,2\delta)$-quasi-geodesic, and so
the distance in $X'$ between $1$ and $\pi(g)$ is at least
\[ \frac{3}{7}(4 I +20\delta)-2\delta > I.\]
It follows that $|g|_X>I$, a contradiction.
\end{proof}

\subsection{Height decreases under filling}\label{ss:heightdecrease}
This subsection is devoted to proving that given a relatively
quasi-convex subgroup $H$, its height decreases after any
sufficiently long $H$-filling.  Our method is the same as the one used
by the second and third authors
for the results in \cite[Part 2]{rhds}; as such it is inspired by
certain hyperbolic $3$-manifold arguments by Lackenby \cite{lackenby:whds} and
by the first author \cite{agol:dehnfilling}.
Briefly, we choose some minimal
counterexample to the theorem, and derive a contradiction by using
``area'' estimates coming from ``pleated surfaces''.

\begin{theorem}\label{t:heightdecrease}
Let $G$ be a torsion-free hyperbolic group, $H<G$ a quasi-convex subgroup of height
$k$, and let
$\mc{P}=\{P_1,\ldots,P_m\}$ be the peripheral structure on $G$ induced by
$H$.  There is a finite set $F\subset G$ so that 
if $G\stackrel{\pi}{\longrightarrow} G(N_1,\ldots,N_m)$ is a
hyperbolic $H$-filling satisfying:
\begin{enumerate}
\item $N_i\lhd P_i$ is finite index for all $i$, and
\item $N_i\cap F=\emptyset$ for all $i$,
\end{enumerate}
then $\pi(H)$ has height strictly less than $k$ in $G(N_1,\ldots,N_m)$.
\end{theorem}
\begin{proof}
Suppose $H<G$ is the height $k$ quasi-convex subgroup and that $H'<G'$
is the image after filling along finite index subgroups of the
malnormal core of $H$.  The filling map from $G$ to $G'$ will be
called $\pi$.  The kernel $N$ of $\pi$ will be normally generated by some
collection of filling kernels $N_1,\ldots,N_m$, each normal in some
element of the peripheral structure on $G$ induced by $H$, and each
contained in $H$.  

We must show that if some intersection of
conjugates of $H'$ is infinite, then this
intersection can be lifted back up to $G$.  In other words, we will
show that infinite
order elements of the intersection of essentially distinct conjugates
of $H'$ are always images of infinite order elements of the
intersection of essentially distinct conjugates of $H$.

We choose some (compatible) generating set $S$ for $G$ so that
$X=X(G,\mc{P},S)$ is $\delta$-hyperbolic.  With respect to this
choice, $H$ is $\lambda$-relatively quasi-convex for some $\lambda$.  By
rechoosing $\delta$, we can assume that $\lambda < \delta$.  We also
assume, as in \cite{rhds} that $\delta$ is an integer greater than or
equal to $100$.  All
constants from \cite{rhds} will be in terms of this re-chosen $\delta$
for $X(G,\mc{P},S)$.  In particular, the auxiliary constants $K =
10\delta$, $L_1=1000\delta$ and $L_2 = 3000\delta$ will be used in
the argument below.

Let $C_h= K + 12\delta + 9$ be the upper bound on the Hausdorff distance
between a geodesic and the preferred path with the same endpoints, from 
\cite[Corollary 5.12]{rhds}.  
For each $i$ let $F_i$ be the ball of $P_i$-radius 
$2^{L_2}(24\cdot 2^{4 C_h + 3}+24)$ about $1$ in $P_i$, and let $F = (\cup_i
F_i)\smallsetminus \{1\}$.

Now fix a hyperbolic $H$-filling
$G\stackrel{\pi}{\longrightarrow}\bar{G}$ of $G$ so that for each $i$,
$N_i$ is finite index in $P_i$ and so that $N_i\cap F = \emptyset$ for
each $i$.
In other words,
for every nontrivial $n$ in any $N_i$, the length in $P_i$ satisfies
\begin{equation}\label{long} 
|n|_{P_i} >  2^{L_2}(24\cdot 2^{4 C_h + 3}+24) .
\end{equation}

Let $N = \ker \pi$.
Any element $g\in N$ is a product in $G$ of conjugates of
elements of $\cup_i N_i$.  We define the \emph{$N$-area} as the smallest number of
such conjugates needed to write $g$.

In order to derive a contradiction, assume that $\pi(H)$ has height at
least $k$.  Thus $k$ essentially distinct conjugates of $\pi(H)$
intersect in an infinite (and quasi-convex) subgroup of
$\bar{G}$.  This subgroup is infinite and hyperbolic, and so
it must contain an element of infinite order.  It follows that
$\pi(H)$ contains infinite order elements $a$ and $b_i$ for 
$i\in\{2,\ldots,k\}$ and essentially distinct $\{1,g_2,\ldots,g_k\}$ so
that
\begin{equation} \label{annulus}
 a = g_i b_i g_i^{-1}                             
\end{equation}
for each $i$.

Fix a lift $\twid{a}$ of $a$ closest to $1$ in $X$, subject to
the condition that $\twid{a}\in H$.   Now for each $i\in
\{2,\ldots,k\}$ choose some $\twid{g}_i$ and $\twid{b}_i$ subject to
the conditions:
\begin{enumerate}
\item[(C1)] $\twid{g}_i\in \pi^{-1}(g_i)$,
\item[(C2)] $\twid{b}_i\in \pi^{-1}(b_i)\cap H$, and
\item[(C3)] $\twid{a}^{-1} \twid{g}_i \twid{b}_i\twid{g}_i^{-1}$ has
  minimal $N$-area for all choices of $\twid{g}_i$ and
  $\twid{b}_i$ satisfying conditions (C1) and (C2).
\end{enumerate}

\begin{claim}\label{claim:zeroarea}
For each $i$, $\twid{a}^{-1} \twid{g}_i \twid{b}_i\twid{g}_i^{-1}$ has
$N$-area zero.
\end{claim}

If the word $\twid{a}^{-1} \twid{g}_i \twid{b}_i\twid{g}_i^{-1}$ has
$N$-area zero, then it is equal in $G$ to $1$.  Therefore, the claim
implies that $\twid{a} = \twid{g}_i \twid{b}_i\twid{g}_i^{-1}$ in $G$
for each $i$.

The claim implies the theorem: Certainly, the elements
$\{1,\twid{g}_2,\ldots,\twid{g}_k\}$ are essentially distinct in $G$.
Because each $\twid{b}_i$ is in $H$, 
the conjugates $H,H^{\twid{g}_2},\ldots,H^{\twid{g}_k}$ all
contain the element $\twid{a}$, and so $H\cap H^{\twid{g}_2}\cap\cdots
\cap H^{\twid{g}_k}$ is infinite.  Since $H$ has height $k$, the
subgroup $H\cap H^{\twid{g}_2}\cap\cdots \cap H^{\twid{g}_k}$ 
is a conjugate of $P_j$ for some $j$.  Since $N_j$ was chosen to be
finite index in $P_j$, this implies that $a=\pi(\twid{a})$ has finite
order in $\bar{G}$, contradicting the original choice of
$a$.

\begin{proof}[Proof of Claim \ref{claim:zeroarea}:]
If the equation $\twid{a}^{-1}
\twid{g}_i \twid{b}_i\twid{g}_i^{-1}$ has $N$-area $p>0$, then there
is some equation
\begin{equation}\label{puncdisk}
\twid{a}^{-1} \twid{g}_i \twid{b}_i\twid{g}_i^{-1} = \prod_{j=1}^p
\alpha_j n_j \alpha_j^{-1}
\end{equation}
with each $n_j\in N_{k_j}$ for some $k_j$.  This equality can be
represented by a punctured disk (as in \cite[Part 2]{rhds}) with
boundary labelled $\twid{a}^{-1} \twid{g}_i
\twid{b}_i\twid{g}_i^{-1}$.  There are two sub-segments of the boundary
of this disk labelled $\twid{g}_i$ and $\twid{g}_i^{-1}$.  Gluing
these together yields an annulus, again with $p$ punctures, as in
Figure \ref{f:bliftnoT}.  
\begin{figure}[htbp]
\begin{center}
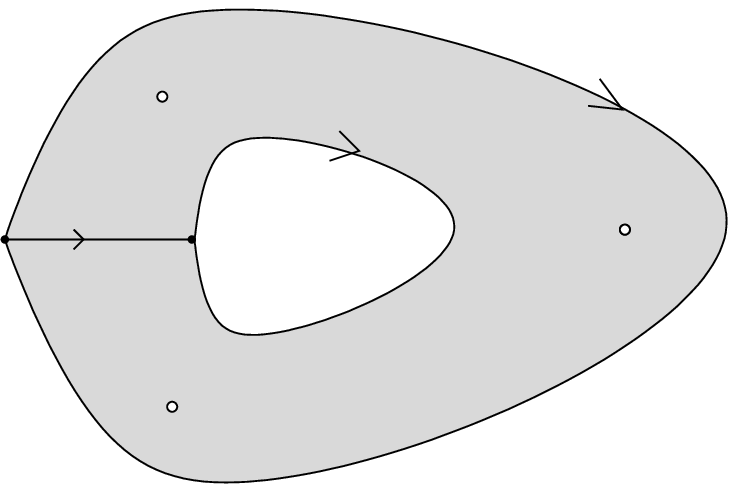
\caption{The punctured annulus $\check{\Sigma}$.}
\label{f:bliftnoT}
\end{center}
\end{figure}
Again as
in \cite[Part 2]{rhds}, there is a proper map from this
punctured annulus into $X/G$, so that labelled sub-segments of the
boundary go to loops representing those elements of $G$ described by
their labels.  The distinguished arc labelled $\twid{g}_i$ is sent by
$\phi$ to a loop representing $\twid{g}_i$.  To be
consistent with the notation of Part 2 of \cite{rhds}, we refer to the
punctured annulus as $\check{\Sigma}$, and the proper map to $X/G$ as
$\check{\phi}$.

We define a \emph{reducing arc} for $\check{\phi}$ to be a proper,
essential embedding $\sigma\co \R\to \check{\Sigma}$, so that
$\check{\phi}\circ\sigma$ can be properly homotoped to miss any given
compact subset of $X/G$.

\begin{subclaim} \label{subclaim:noRA}
There are no reducing arcs for $\check{\phi}$.
\end{subclaim}
\begin{proof}
We argue by contradiction.  Suppose that $\sigma\co\R\to\check{\Sigma}$ is
a reducing arc.  

\noindent{\bf Case 1:} We suppose first that the closure of the image
of $\sigma$ in the unpunctured annulus is a homotopically nontrivial
loop. In this case, the surface $\check{\Sigma}$ may be cut along the
image of $\sigma$ to yield a pair of surfaces.  The half of the
surface with boundary labelled $\twid{a}$ represents a proof that
$a=\pi(\twid{a})$ is parabolic in $\bar{G}$.  Since every parabolic
element of $\bar{G}$ has finite order, this is a contradiction.

\noindent{\bf Case 2:} Now suppose that the image of $\sigma$ does not
intersect the distinguished arc in $\check{\Sigma}$ labelled
$\twid{g}_i$.  We may argue very much as in the proof of Claim 9.2 in
the proof of Theorem 9.1 of \cite{rhds}:  The case that $\sigma$
connects two distinct punctures is the same as Case 1 of that argument, and
the case that $\sigma$ connects a puncture to itself is the same as
Case 4 of that argument, except
that
if the reducing arc $\sigma$ represents a peripheral element which
is not in $N_i$ then the contradiction is to the conclusion of Theorem 9.1
of \cite{rhds},
rather than to the minimality of the diagram.  Since 
\[2^{L_2}(24\cdot 2^{4 C_h + 3}+24) > 12\cdot 2^{L_2},\]
the conclusion of Theorem 9.1 of \cite{rhds} \emph{does} hold.
Cases 2 and 3 of the proof from \cite[Claim 9.2]{rhds} do not occur.
The upshot here is that if a reducing arc appeared, we
would be able to choose a new expression of the form \eqref{puncdisk}
with smaller $N$-area by performing a ``boundary reduction'' of some
kind.

\noindent{\bf Case 3:} Finally, we suppose that the image of $\sigma$
intersects the arc labelled $\twid{g}_i$, but we are not in Case 1.  

There is a natural basepoint $\bar{1}$ in $X/G$, which is the image of 
$1 \in G \subset X$ under the quotient map (and also, of course, the image
of any other group element). There is a canonical identification between
$\pi_1(X/G,\bar{1})$ with $G$.
Let $\gamma : I \to \check{\Sigma}$ be any arc with the same endpoints as the
distinguished arc labelled by $\twid{g}_i$.  Then $\check{\phi} \circ \gamma$ is
a loop in $X/G$ based at $\bar{1}$ and so determines a unique element
$g_\gamma$ of $G$.  The identity
\[	a^{-1} \pi(g_\gamma) b_i \pi(g_\gamma) = 1	,	\]
always holds in $\bar{G}$, and has $N$-area at most $p$.

Since we are not in Case 1, there is an arc $\gamma$ in $\check{\Sigma}$,
with the same endpoints as the distinguished arc, which does not intersect
$\sigma$.  We are now in Case 2.

This completes the proof of Subclaim \ref{subclaim:noRA}.
\end{proof}

Choose a (partially ideal) triangulation $\mc{T}$ of the punctured annulus $\check{\Sigma}$
whose vertices are the endpoints of the distinguished arc and the punctures.  There
are $2p+2$ triangles in such a triangulation.

Since there are no reducing arcs, the hypotheses of \cite[Lemma 8.6]{rhds} are
satisfied.  Let
\[\straightedge{\extend{\phi}}\co\extend{\Sigma}\to X/K\]
be the map from \cite[Lemma 8.6]{rhds}, which sends each edge of
$\mc{T}$ to a preferred path, and let 
\[ \skelmap \co \skelfill \to X/G \cup (\Hbound)/G, \] 
be the map from \cite[Remark 8.11]{rhds}, where it is called
$\ddot{\check{\phi}}_{\mathcal{T}}$. (Elements of $(\Hbound)/G$ are
$G$-orbits of horoball centers, and are in one to one correspondence
with $\mc{P}$.)
We note some facts about the skeletal filling
$\skelfill$ and the map
$\skelmap$:
\begin{enumerate}
\item If $\bar{\Sigma}$ is $\check{\Sigma}$ with punctures filled in,
  then the $1$-complex $\skelfill$ embeds naturally in $\bar{\Sigma}$
  so that every edge is either
  \begin{enumerate}
  \item part of one of the edges from $\mc{T}$,
  \item a \emph{rib} (with image under $\straightedge{\extend{\phi}}$ a
    horizontal edge at depth $L_2$), or
  \item a \emph{ligament} (with image under
    $\straightedge{\extend{\phi}}$ a vertex at depth $L_2$),
  \end{enumerate}
  and every vertex is either
  \begin{enumerate}
  \item coincident with a filled-in puncture of $\bar{\Sigma}$, 
  \item a vertex of $\mathcal{T}$, or
  \item the endpoint of one or two ribs or ligaments.
  \end{enumerate}
  The last two types of vertices will be called \emph{ordinary} vertices.
\item Every vertex $v$ coming from a puncture has a \emph{link}, which
  is a circle in $\skelfill$ composed of ribs, ligaments, and possibly
  some sub-segments of edges of $\mc{T}$ in the boundary of
  $\check{\Sigma}$. (See \cite[Definition 8.12]{rhds} for the precise
  definition.)  The map $\skelmap$ sends the entire link to that part
  of $X/G$ at depth $L_2$ or more.  A puncture is called
  \emph{interior} if its link is composed entirely of ribs and
  ligaments.  Otherwise it is called \emph{exterior}.
\item Every path between ordinary vertices is sent by $\skelmap$ to a
  path in $X/G$ which is either a based loop at $\bar{1}$, or can be
  made into one in a canonical way by adding vertical segments.  Thus
  any path between ordinary vertices in $\skelfill$ determines an
  element of $G$.
\item\label{concat} The group element determined by the concatenation
  of paths between ordinary vertices in $\skelfill$ is the product of
  those determined by the paths.
\item\label{kernel} The group element determined by a loop around the
  link of a vertex is always an element of $N_i$ for some $i$.
  Different choices of starting point for the loop give rise to
  elements of $N_i$ which are conjugate in $P_i$.
\item Two paths homotopic in $\bar{\Sigma}$ rel their endpoints
  determine the same element of $\bar{G}$. (This follows immediately
  from \eqref{concat} and \eqref{kernel}.)
\end{enumerate}

Figure \ref{f:blift} shows an example of what $\skelfill$ might look
like, if there were three exterior and no interior punctures.
\begin{figure}[htbp]
\begin{center}
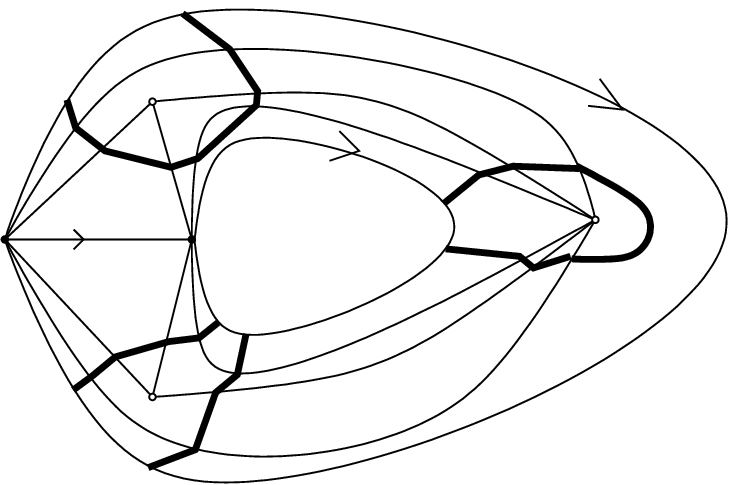
\caption{Dark edges are either ribs or ligaments.  Hollow circles are
  vertices of $\skelfill$ coming from punctures.  The (exterior)
  puncture at the right has a link composed of $7$ ribs or ligaments,
  and a single edge which is part of the edge of $\mc{T}$ labelled by
  $\twid{b}_i$.}
\label{f:blift}
\end{center}
\end{figure}

If $T$ is a $2$-simplex of $\mc{T}$, then $\phi|_{\partial T}$ lifts to a preferred
triangle $\widetilde{\phi|_{\partial T}}\co\partial T\to X$.  Let $R(T)$ be the number
of ribs in $\skel(\widetilde{\phi|_{\partial T}})$, and note that this number does not
depend on the lift chosen.
\cite[Corollary 5.38]{rhds} implies that $R(T)\leq 6$.

Let 
\[A(\phi) = \sum_{T\in \mathcal{T}} R(T).\]
\cite[Corollary 5.38]{rhds} immediately implies
\begin{equation}\label{upperbound}
A(\phi)\leq 6 (2p+2) \le 24p.
\end{equation}

Let $x$ be a puncture.  As remarked above, the link $\Lk (x)$ is an
embedded loop in $\check{\Sigma}$ whose image in $X/G$ lies entirely
at or below depth $L_2$.  By joining an arbitrarily chosen basepoint
of $\Lk (x)$ to $\bar{1}$ in the canonical way, we obtain an element
of the filling kernel $N_i$, contained in the peripheral subgroup
$P_i$ for some $i$.  

\begin{subclaim}\label{sc:notallint}
Not all punctures of $\check{\Sigma}$ are interior.
\end{subclaim}
\begin{proof}
Suppose that all punctures of $\check{\Sigma}$ were interior.
In this case, the image of each link $\Lk (x)$ lies entirely at depth $L_2$ in $X/G$,
and represents a conjugacy class of element of $N_i$ for some $i$.  
The length of $\check{\phi} (\Lk (x))$ is the number of ribs in $\Lk (x)$.  Therefore
there are at least $24\cdot 2^{4C_h + 3} + 24$ ribs in each link.
Summing over all the links, there are more than $24p$ ribs, which contradicts
\eqref{upperbound} above.

This proves that not all punctures are interior.
\end{proof}

\begin{subclaim}\label{sc:nob}
No link of a puncture hits the boundary
component of $\check{\Sigma}$ labelled $\twid{b}_i$.  
\end{subclaim}
\begin{proof}
By way of contradiction, we assume that there is some puncture $x$ so
that $\Lk (x)$
intersects the boundary
component of $\check{\Sigma}$ labelled $\twid{b}_i$ in a sub-segment $I$.
The idea here is that
if it did, we would be able to replace $\twid{b}_i$ by some
$\twid{b}_i'$ so that $\twid{a}^{-1} \twid{g}_i
\twid{b}_i\twid{g}_i^{-1}$ had smaller $N$-area.  Our assumption
that $\lambda$ is small with respect to $\delta$ (and therefore with
respect to $L_2$) ensures that this new lift $\twid{b}_i'$ still lies
in $H$.

We will choose $\twid{b}_i'$ to be the element determined by the path
in $\skelfill$ which is obtained from the path
labelled $\twid{b}_i$ by replacing $I$ with its complement in $\Lk(x)$.
Let $\beta$ be the group element represented by the part of the edge
labelled by $\twid{b}_i$ which precedes $I$, and let $n\in N_j$ be the
element represented by the loop around $\Lk(x)$ starting at the
beginning of $I$ and going around once, clockwise.  It is clear that 
\[\twid{b}_i' = \beta n \beta^{-1} \twid{b}_i \]
maps to $b_i$, and has $N$-area less than that of $\twid{b}_i$.  To
establish the subclaim, it
remains only to establish that $\twid{b}_i'\in H$.

The restriction of the map $\straightedge{\extend{\phi}}\co
\check{\Sigma}\to X/G$ to the edge $e$
of $\check{\Sigma}$ labelled $\twid{b}_i$ can be lifted to a preferred
path $\gamma\co e \to X$ joining $1$ to $\twid{b}_i$.  The map
$\gamma$ sends the
sub-segment $I$ into the $L_2$-horoball inside some $0$-horoball $A$,
corresponding to some coset $tP_j$ of one of the peripheral subgroups.  

As in the proof of Lemma \ref{l:technical2}, we first argue that some
nontrivial element of $H$ stabilizes the horoball $A$.  The preferred
path from $1$ to $\twid{b}_i$ must penetrate at least $L_2$ into the
horoball $A$.  It follows that a regular geodesic $\gamma$ penetrates
$A$ at least to depth $L_2-C_h > 2500\delta $.  Let $g_1$ and $g_2$ be
the group elements in $tP_j$ through which this geodesic passes.  By
the foregoing, we must have $d_X(g_1,g_2) > 5000\delta$.  Moreover,
the geodesic $\gamma$ passes through the vertices $(g_1,\lambda+1)$
and $(g_2,\lambda+1)$ in $A$.  The $\lambda$-relative quasi-convexity
of $H$ implies that there are vertices $z_1=\check\iota\left((s_1
D_{j_1},h_1,n_1)\right)$ and $z_2=\check\iota\left((s_2
D_{j_2},h_2,n_2)\right)$  within $\lambda$ of $(g_1,\lambda+1)$ and
$(g_2,\lambda+1)$, respectively.  Since $g_1$ and $g_2$ are so far
apart (recall we have made the assumption $\lambda\leq \delta$), the
element $h_2h_1^{-1}$ is non-trivial.  By the $H$-equivariance of
$\check\iota$, this element stabilizes $A$, i.e. $H\cap tP_jt^{-1}$ is
non-trivial.

The element $\beta = t p $ for some $p\in P_j$, and so 
\begin{eqnarray*}
 \twid{b}_i' & = & \beta n \beta^{-1} \twid{b}_i \\
 & = & t (p n p^{-1}) t^{-1} \twid{b}_i.
\end{eqnarray*}
Since $N_j$ is normal in $P_j$, the element $pnp^{-1}\in N_j$.
Because of the assumption that $\bar{G}$ is an $H$-filling of $G$, 
the subgroup $tN_jt^{-1}$ lies in $H$.  It follows that
$\twid{b}_i'\in H$, completing the proof of the subclaim.
\end{proof}

By Subclaim \ref{sc:notallint}, some puncture or punctures are exterior; by Subclaim
\ref{sc:nob}, the links of the exterior punctures all intersect the
edge of $\mc{T}$ labelled $\twid{a}$, and miss the edge of $\mc{T}$
labelled $\twid{b}_i$.
We will
show that if any link of a puncture
hits the side labelled $\twid{a}$, we can find another lift of $a$
whose $X$-length is smaller, contradicting our initial choice of a
shortest lift.

Let $r_j$ be the number of ribs in the link of the $j$'th puncture.
We have
\[\sum_j r_j \leq 24 p.\]
Also associated to the puncture is a sub-segment of the preferred path
from $1$ to $\twid{a}$ passing through an $L_2$-horoball.
Specifically, it passes through some points $(x_j,L_2)$ and
$(y_j,L_2)$ in a horoball based on $P_{k_j}$.  Let $q_j$ be the
distance in the $L_2$-horosphere between $v_j=(x_j,L_2)$ and $w_j=(y_j,L_2)$.
The ribs in the link of the $j$'th puncture give an edge-path in this
horosphere from $(x_j,L_2)$ to $w_j' =(y_j z_j,L_2)$ for some $z_j\in N_{k_j}$.
We therefore have
\[r_j+q_j \geq 2^{-L_2} C_N\]
for each $j$.
(If we are looking at an interior puncture, we have $q_j = 0$.)
Thus 
\[\sum_j r_j+q_j \geq 2^{-L_2} C_N p \]
and 
\[\sum_j q_j \geq (2^{-L_2} C_N - 24) p.\]

We now claim that for some $j$, 
\begin{equation}\label{somej}
q_j/r_j > 2^R.
\end{equation}  
Indeed, if this is
never the case, then
\[ \sum_j q_j \leq 2^R \sum_j r_j \leq 24 \cdot 2^R p,\]
which implies that 
\[ 2^{-L_2} C_N - 24 \leq 24\cdot 2^R, \]
and so 
\[ C_N \leq 2^{L_2} ( 24 \cdot 2^R + 24 ), \]
contradicting \eqref{long}.

Choose some such $j$, and consider the lift $\twid{a}'$ of $a$
obtained by excising the word representing $x_j^{-1} y_j$ and
replacing it by a word representing $x_j^{-1} y_j z_j$.
The distance from $1$ to $\twid{a}'$ is at most
\[ d(1,v_j)+d(v_j,w_j')+d(w_j,\twid{a}'), \]
which is at least $(\log_2(q_j)-\log_2(r_j)) - 4 C_h$ less than the
distance from $1$ to $\twid{a}$ (for this probably a picture should be
drawn).  Equation \eqref{somej} implies that $\twid{a}'$ is actually
shorter than $\twid{a}$, contradicting our initial choice of
$\twid{a}$.

This completes the proof of Claim \ref{claim:zeroarea}.
\end{proof}

By Claim \ref{claim:zeroarea}, there are lifts $\twid{a}$,
$\twid{g}_i$, $\twid{b}_i$ of $a$, $g_i$ and $b_i$ respectively
satisfying the conditions (C1--C3) and with
\[ \twid{a} = \twid{g}_i \twid{b}_i \twid{g}_i^{-1} \]
for each $i\in \{2,\ldots,k\}$.  Since $\{1,g_2,\ldots,g_k\}$ are
essentially distinct, so are the lifts
$\{1,\twid{g}_2,\ldots,\twid{g}_k\}$, and so $\twid{a}$ lies in the
intersection of $k$ essentially distinct conjugates of $H$.  
Since $\twid{a}$
has infinite order, and $H$ has height $k$,
it follows that $\twid{a}\in P_j$ for some $j$.  But since $N_j$ has finite
index in $P_j$, the image $a$ of $\twid{a}$ in $\bar{G}$ must have
finite order.  This contradiction completes the proof of Theorem
\ref{t:heightdecrease}.
\end{proof}

\subsection{The proof of Theorem \ref{t:technical}}

Let $G$ be a torsion-free residually finite hyperbolic group, let
$H$ be a quasi-convex subgroup of $G$ of height $k$ and let $g \in G \smallsetminus H$. 

Let $\mc{D}$ be the malnormal core of $H$, and let $\mc{P}$ be the
peripheral structure on $G$ induced by $H$.

By Proposition \ref{p:induced}, $G$ is hyperbolic relative to $\mc{P}$,
$H$ is hyperbolic relative to $\mc{D}$ and $H$ is $\lambda$-relatively
quasi-convex in $G$ for some $\lambda$.

Since all the elements of $\mc{P}$ are subgroups of $G$, they are
residually finite.  Thus they contain finite-index normal subgroups
$\{ N_i \}$ which induce an $H$-filling of $G$ satisfying the
hypotheses of Propositions \ref{p:fillqc}, \ref{p:nocollide} and 
Theorem \ref{t:heightdecrease}.

We claim that the group $\bar{G} = G(N_1,\ldots,N_m)$ satisfies the 
conclusion of Theorem \ref{t:technical}. Let $\eta : G \to \bar{G}$ be the
canonical quotient map.  By Proposition \ref{p:fillqc}, $\eta(H)$ is
relatively quasi-convex.  Since the peripheral subgroups of $\bar{G}$
are finite, Lemma \ref{l:finiteparabolics} implies that $\eta(H)$ is
actually quasi-convex in the hyperbolic group $\bar{G}$.
Further, $\eta(g) \not\in \eta(H)$ by Proposition \ref{p:nocollide} and
the height of $\eta(H)$ is at most $k-1$ by Theorem \ref{t:heightdecrease}.

\section{Conclusion}

It would be nice to extend the main result of this paper to groups which
are residually hyperbolic, that is  to groups $G$ which for any element
$1\neq g\in G$, there is a homomorphism $\varphi: G \to H$ onto a hyperbolic
group such that $\varphi(g)\neq 1$. A natural class of such groups
are groups $G$ which are relatively hyperbolic, relative to
a finite family
$\mc{P}=\{P_1, \ldots, P_n\}$ of finitely
generated residually finite subgroups of $G$.
If hyperbolic groups are residually finite, then these groups
are also residually finite, by performing finite fillings on the peripheral
subgroups $\mc{P}$ using Theorem \ref{t:rhds}. 
To generalize Theorem \ref{t:main} to this class
of groups, we would need to identify the analogue of quasi-convex subgroups
to separate. It would be natural to consider relatively quasi-convex subgroups
of $G$ (as in Definition 3.10). To be somewhat conservative, though, we will
consider the case that the groups $P_i$ are nilpotent. 

\begin{conjecture}
Suppose that hyperbolic groups are residually finite. Let $G$ be a group
which is hyperbolic relative to the peripheral system 
$\mc{P}=\{P_1,\ldots, P_n\}$. 
If $P_i$ is finitely generated and virtually nilpotent for each $i$,
then relatively quasi-convex subgroups 
of $G$ are separable. 
\end{conjecture}

This conjecture would extend Theorem \ref{t:main} to non-uniform lattices
in rank one symmetric spaces. This conjecture would also suffice to prove
that if hyperbolic groups are RF, then Kleinian groups are LERF. 

\small
\def\cprime{$'$} \providecommand\url[1]{\texttt{#1}}

\end{document}